\documentclass{amsart}

\usepackage[all]{xy}
\usepackage{graphicx}
\usepackage{amssymb}
\usepackage{mathrsfs}
\usepackage{multirow}
\usepackage{float}
\usepackage{tikz-cd}
\usepackage{adjustbox}
\usepackage{amsthm}
\usepackage{accents}
\usepackage{mathtools}
\usepackage{enumitem,cite,array}
\usepackage[new]{old-arrows}
\usepackage{tikz}
\usetikzlibrary{matrix,arrows}
\usepackage{amsmath}

\usepackage[numbers,sort&compress]{natbib}
\usepackage[bookmarksnumbered, bookmarksopen,
colorlinks,citecolor=blue,linkcolor=blue,backref]{hyperref}

\newcounter{RomanNumber}

\newtheorem{theorem}{Theorem}%[section]

\theoremstyle{definition}

\newtheorem{proposition}[theorem]{Proposition}
\newtheorem{corollary}[theorem]{Corollary}

\theoremstyle{remark}

\theoremstyle{notation}

\newcommand{\be}{\begin{equation}}
\newcommand{\ee}{\end{equation}}

\newcommand{\CC}{\mathbb{C}}

\newcommand{\Z}{\mathbb{Z}}

\newcommand{\N}{\mathbb{N}}

\newcommand{\h}{\begin{eqnarray*}}
\newcommand{\e}{\end{eqnarray*}}

\newcommand{\ep}{\varepsilon}

\usepackage[utf8]{inputenc}
\usepackage{chngcntr}
\usepackage{apptools}
\AtAppendix{\counterwithin{theorem}{section}}
\usepackage[toc,page]{appendix}

\newcommand{\qqed}{\hfill\Box}
\begin{document}


\title
{Elliptic genus and string cobordism at dimension $24$}

\author{Fei Han}
\address{Department of Mathematics,
National University of Singapore, Singapore 119076}
\curraddr{}
\email{mathanf@nus.edu.sg}
\thanks{}
\urladdr{https://blog.nus.edu.sg/mathanf/}

%author two information

 \author{Ruizhi Huang}
\address{Ruizhi Huang, Institute of Mathematics and Systems Sciences, Chinese Academy of Sciences, Beijing 100190, China}
\email{huangrz@amss.ac.cn}  
\urladdr{https://sites.google.com/site/hrzsea}
\thanks{}

%    \subjclass is required.
%\subjclass[2010]{Primary 58J26, Secondary 57S20, 53C27, 11F55, 19K35}
%\keywords{Elliptic genera, Witten genus, proper actions, loop Dirac induction, modularity, group $C^*$-algebras, representation ring, operator K-theory}
\date{}
%\thanks{}

\maketitle

\begin{abstract} It is known that spin cobordism can be determined by Stiefel-Whitney numbers and index theory invariants, namely $KO$-theoretic Pontryagin numbers. In this paper, we show that string cobordism at dimension 24 can be determined by elliptic genus, a higher index theory invariant. We also compute the image of 24 dimensional string cobordism under elliptic genus. Using our results, we show that under certain curvature conditions, a compact 24 dimensional string manifold must bound a string manifold.  

\end{abstract}

%\tableofcontents

%%%%%%%%%%%%%%%%%%%%%%%%%%%%%%%%%%%%%%%%%%%%%%%%%%%%%%%%%%%%%%%%%%%%%%%
%%%%%%%%%%%%%%%%%%%%%%%%%%%%%%%%%%%%%%%%%%%%%%%%%%%%%%%%%%%%%%%%%%%%%%%
\section{Introduction}
Cobordism is a fundamental tool in geometry and topology. For the oriented cobordism ring $\Omega_\ast^{SO}$, there are spin cobordism $\Omega_\ast^{Spin}$ and string cobordism $\Omega_\ast^{String}$ as refinements through the Whitehead tower
\[
\cdots \longrightarrow String \longrightarrow Spin\longrightarrow SO.
\] 

It is a classical problem to classify cobordism classes in terms of characteristic numbers. Historically, Wall \cite{Wall60} showed that two closed oriented manifolds are oriented cobordant if and only if they have the same Stiefel-Whitney numbers and Pontryagin numbers. Later, Anderson, Brown and Peterson \cite{ABP67} showed that two closed spin manifolds are spin cobordant if and only if they have the same Stiefel-Whitney numbers and $KO$-theoretic characteristic numbers. 

The problem for string manifolds is much more complicated. To our best knowledge, it is unknown yet which set of characteristic numbers classifies string cobordism. It is expected that $TMF$-theoretic characteristic numbers will play the similar role for string cobordism as $KO$-theoretic characteristic numbers does for spin cobordism. Here $TMF$ stands for the topological modular form developed by Hopkins and Miller \cite{Hop02}. The Witten genus \cite{Wit87, Wit88} plays the similar role in $TMF$ as the $\widehat{A}$-genus does in $KO$ and is refined to be the $\sigma$-orientation from the Thom spectrum of  string cobordism to the spectrum $TMF$ \cite{Hop02, AHS01}. 

In this paper we show that the elliptic genus \cite{Och87}, a higher index theoretic invariant, determines 24 dimensional string cobordism. As elliptic genus is a twisted Witten genus \cite{Wit87, Wit88, Liu95, Liu95JDG}, it can be viewed as sort of $TMF$-theoretic characteristic numbers. This coincides with the expectation of the role that $TMF$-theoretic characteristic numbers should play for string cobordism. In the paper, we also compute the image of 24 dimensional string cobordism under elliptic genus as well as give some application of our results in geometry. It is worthwhile to remark that $24$ is a dimension of special interest for string geometry. For instance, in this dimension, one has (cf. page $85$-$87$ in \cite{HBJ94})
\[
W(M)=\widehat{A}(M)\bar{\Delta}+\widehat{A}(M,T)\Delta,
\] where $W(M)$ is the Witten genus of $M$, $\widehat{A}(M)$ is the A-hat genus and $\widehat{A}(M,T)$ is the tangent bundle twisted A-hat genus of $M$, 
$\bar{\Delta}=E_4^3-744\cdot \Delta$
with $E_4$ being the Eisenstein series of weight $4$ and $\Delta$ being the modular discriminant of weight $12$. Hirzebruch raised his prize question in \cite{HBJ94} that whether there exists a $24$ dimensional compact string manifold $M$ such that $W(M)=\bar{\Delta}$ (or equivalently $\widehat{A}(M)=1, \widehat{A}(M, T)=0$) and the Monster group acts on $M$ as self-diffeomorphisms.
The existence of such manifold was confirmed by Mahowald-Hopkins \cite{MH02}. They determined the image of Witten genus at this dimension via $TMF$. Recently Milivojevi\'c \cite{Mili} used rational homotopy theory to give a weak form solution to the Hirzebruch's prize question. However, the part of the question concerning the Monster group action is still open.

The {\it elliptic genus}, which was first constructed by Ochanine \cite{Och87} and Landweber-Stong \cite{LS88}, is a graded ring homomorphism
\begin{equation}\label{elleq}
\phi: \Omega_\ast^{SO}\longrightarrow \mathbb{Z}\big[\frac{1}{2}\big] [\delta, \varepsilon]
\end{equation} from the oriented cobordism ring to the graded polynomial ring  $\mathbb{Z}\big[\frac{1}{2}\big] [\delta, \varepsilon]$ with the degrees $|\delta|=4$, $|\ep|=8$, such that the logarithm is given by the formal integral
\begin{equation}\label{logeq}
g(z)=\int_{0}^{z}\frac{dt}{\sqrt{1-2\delta t^2 +\varepsilon t^4}}.
\end{equation}
The background and the developments of the theory of elliptic genus can be found in  \cite{Lan88-1, Se88, KS93, HBJ94, Liu96, Hop02, Wit88}. 

It is shown (\cite{CCLOS}, c.f.\cite{Lan88}) that the image of the elliptic genus  is
\be \label{oriented}  \phi(\Omega_\ast^{SO})=\mathbb{Z}[\delta, 2\gamma, 2\gamma^2, \cdots, 2\gamma^{2^s}, \cdots], \ee
where $\gamma=\frac{\delta^2-\varepsilon}{4}$; and when restricted to spin cobordism
\be \phi(\Omega_\ast^{Spin})=\mathbb{Z}[16\delta, (8\delta)^2, \varepsilon].\ee
It follows that at dimension $24$ the image is spanned over $\mathbb{Z}$ by 
$$(8\delta)^6, \ (8\delta)^4\varepsilon,  \ (8\delta)^2\varepsilon^2, \ \varepsilon^3. $$
The map $\phi: \Omega_{24}^{Spin}\to \mathbb{Z}[8\delta, \varepsilon]$ has nontrivial kernel. Actually $E-F\cdot B$ is in the kernel, where $E$ is the total space of a fiber bundle of compact and connected structure group with $F$ being spin manifold as fiber and $B$ being the base. This comes from the multiplicativity of elliptic genus \cite{Och88}, which is  equivalent to the Witten-Bott-Taubes-Liu rigidity \cite{BT, Tau, Liu96TOP}.

Our main result is stated as follows.

\begin{theorem}\label{ellipticinjthm}
The elliptic genus 
\[
\phi: \Omega_{24}^{String}\longrightarrow \mathbb{Z}[8\delta, \varepsilon]
\]
is injective and its image is a subgroup of  $\mathbb{Z}[8\delta, \varepsilon]$ spanned by 
$$(8\delta)^6, \  \  24 (8\delta)^4 \varepsilon,  \  \ (8\delta)^2\varepsilon^2, \  \ 8\varepsilon^3.$$
\end{theorem}
The theorem shows us the following picture, 
$$\Omega_{24}^{String}\cong  \phi(\Omega_{24}^{String})\cong \Z\oplus 24\Z\oplus \Z\oplus 8\Z\leq \Z\oplus \Z\oplus \Z\oplus \Z\cong \phi(\Omega_{24}^{Spin}).$$
In particular, it supports the expectation that $TMF$-theoretic characteristic numbers will play the similar role for string cobordism as $KO$-theoretic characteristic numbers do for spin cobordism. 

The key to the proof of Theorem \ref{ellipticinjthm} is a result in \cite{HH21}, where we determine an integral basis of $\Omega_{24}^{String}$, which consists of two explicitly constructed manifolds in the kernel of the Witten genus, and another two manifolds constructed by Mahowald and Hopkins \cite{MH02} determining the image of the Witten genus. Then we can apply two concrete elliptic genera (\ref{expression}) to reduce the computations of the elliptic genus to those of classical twisted and untwisted genera on the generators of $\Omega_{24}^{String}$. The details are carried out in Section \ref{sec: pfthm}.

Theorem \ref{ellipticinjthm} has interesting application in geometry. A closed manifold $M$ is called almost flat if for any $\varepsilon >0$,  there is a Riemannian metric $g_\varepsilon$  on $M$ such that the diameter ${\rm diam}(M,g_{\varepsilon })\leq 1$ and $g_\varepsilon$  is $\varepsilon$-flat, i.e. for the sectional curvature $K_{g_{\varepsilon }}$, we have $|K_{g_{\epsilon }}|<\varepsilon$. Given $n$, there is a positive number $\varepsilon _{n}>0$  such that if an $n$-dimensional manifold admits an $\varepsilon _{n}$-flat metric with diameter $\leq1$,  then it is almost flat. The classical result of Gromov \cite{Gro} says that every almost flat manifold is finitely covered by a nilmanifold, and this was refined by Ruh \cite{Ruh} by proving that an almost flat manifold is diffeomorphic to an infranilmanifold. It has been conjectured by Farrell and Zdravkovska \cite{FZ83} and independently by Yau \cite{Yau93} that every almost flat manifold is the boundary of a closed manifold. Davis and Fang \cite{DF16} showed that this conjecture holds under the assumption that the $2$-sylow subgroup of holonomy group is cyclic or generalized quaternionic. The general case of the conjecture remains open. It is also pointed in \cite{DF16} that it is a difficult question that if every almost flat spin manifold (up to changing spin structures) bounds a spin manifold. 

By the Chern-Weil theory, it can be shown that the Pontryagin numbers of an oriented almost flat manifold $M$ all vanish (c.f. \cite{DF16}). Since the elliptic genus is determined by Pontryagin numbers, one can see from Theorem \ref{ellipticinjthm} that every 24 dimensional almost flat string manifold bounds a string manifold. 

In \cite{CH20}, vanishing results for elliptic genus were proven under almost nonpositive Ricci curvature condition. Corollary 2 there shows that 
given $n\in \N$ and positive number $\lambda$, there exists some $\varepsilon =\varepsilon(n, \lambda) > 0$ such that
if a compact $4n$-dimensional spin Rimannian manifold $(M, g)$ satisfies ${\rm diam}(M, g) \leq 1, {\rm Ric}(g)\leq \varepsilon$,
sectional curvature $\geq -\lambda$ and has infinite isometry group, then the elliptic genus of $M$ vanishes. Combining Theorem \ref{ellipticinjthm}, we obtain
\begin{corollary}\label{riccicoro}Given positive number $\lambda$, there exists some $\varepsilon =\varepsilon(\lambda) > 0$ such that
if a compact 24-dimensional string Rimannian manifold $(M, g)$ satisfies ${\rm diam}(M, g) \leq 1, {\rm Ric}(g) \leq \varepsilon,$
sectional curvature $\geq -\lambda$ and has infinite isometry group, then $M$ bounds a string manifold.  $\qqed$
\end{corollary}

\bigskip

\noindent{\bf Acknowledgements.}
Fei Han was partially supported by the grant AcRF R-146-000-263-114 from National University of Singapore. He thanks Prof. Kefeng Liu and Prof. Weiping Zhang for helpful discussions. Ruizhi Huang was supported by National Natural Science Foundation of China (Grant nos. 11801544 and 11688101), and ``Chen Jingrun'' Future Star Program of AMSS.

%%%%%%%%%%%%%%%%%%%%
%%%%%%%%%%%%%%%%%%%%%%%%%%%%%%%%%%%%%%%%%%%%%%%%%%%%%%%%%%%%%%%%%%%%%%%

\numberwithin{equation}{section}
\numberwithin{theorem}{section}
%----------------------------------------------------------------------------------------------------------------------------------------------------------------------------------------------------------%
\section{Preliminaries}
\label{sec: ell}
In this section we collect some necessary knowledge of elliptic genus used in the sequel. Details can be found in \cite{HBJ94}, \cite{Liu92}, \cite{Liu95}. 

Let $f$ be the formal inverse function of the logarithm $g$ in (\ref{logeq}). Then $Y=f', X=f$ solve the Jacobi quadrics
\begin{equation}\label{odefeq}
Y^2=1-2\delta \cdot X^2 +\varepsilon X^4.
\end{equation}
For concrete values of $\delta$ and $\varepsilon$, a solution $f$ gives an elliptic genus with logarithm $g$. For instance, when $\delta=\varepsilon=1$, $f(z)=\tanh z$ and $\phi$ reduces to the $L$-genus or the signature, and when $\delta=-1/8$, $\varepsilon=0$, $f(z)=2\sinh\frac{z}{2}$ and $\phi$ reduces to the $\widehat{A}$-genus.

Recall that the four Jacobi theta-functions (c.f. \cite{Cha85}) defined by
infinite multiplications are
\[
\begin{split}
\theta(v,\tau)&=2q^{1/8}\sin(\pi v)\prod_{j=1}^\infty[(1-q^j)(1-e^{2\pi \sqrt{-1}v}q^j)(1-e^{-2\pi
\sqrt{-1}v}q^j)], \\
\theta_1(v,\tau)&=2q^{1/8}\cos(\pi v)\prod_{j=1}^\infty[(1-q^j)(1+e^{2\pi \sqrt{-1}v}q^j)(1+e^{-2\pi
\sqrt{-1}v}q^j)], \\
 \theta_2(v,\tau)&=\prod_{j=1}^\infty[(1-q^j)(1-e^{2\pi \sqrt{-1}v}q^{j-1/2})(1-e^{-2\pi
\sqrt{-1}v}q^{j-1/2})],\\
\theta_3(v,\tau)&=\prod_{j=1}^\infty[(1-q^j)(1+e^{2\pi
\sqrt{-1}v}q^{j-1/2})(1+e^{-2\pi \sqrt{-1}v}q^{j-1/2})], 
\end{split}
\]
where
$q=e^{2\pi \sqrt{-1}\tau}$.
They are holomorphic functions for $(v,\tau)\in \mathbb{C \times
H}$, where $\mathbb{C}$ is the complex plane and $\mathbb{H}$ is the
upper half plane. Write $\theta_j=\theta_j(0, \tau), \ 1\leq j\leq 3$, and $\theta^\prime(0, \tau)=\frac{\partial}{\partial v}\theta(v, \tau)\big|_{v=0}$.

When
\begin{equation}\label{d1e1eq}
\begin{split}
\delta=\delta_1(\tau) &= \frac{1}{8} (\theta_2^4 +\theta_3^4)=\frac{1}{4}+6\sum\limits_{n=1}^{\infty}\sum\limits_{\substack{d|n \\ d~{\rm odd}}} d q^n=\frac{1}{4}+6q+6q^2+\cdots, \\
\varepsilon=\varepsilon_{1}(\tau)&=\frac{1}{16} \theta_2^4 \theta_3^4=\frac{1}{16}+\sum\limits_{n=1}^{\infty}\sum\limits_{\substack{d|n}} (-1)^d d^3 q^n=\frac{1}{16}-q+7q^2+\cdots,
\end{split}
\end{equation}
the equation (\ref{odefeq}) has the solution
\[
f_1(z, \tau)= 2\pi\sqrt{-1}\frac{\theta(z, \tau)}{\theta^\prime (0,\tau)} \frac{\theta_1(0,\tau)}{\theta_1(z, \tau)}.
\]
Similarly, when
\begin{equation}\label{d2e2eq}
\begin{split}
\delta=\delta_2(\tau)&= -\frac{1}{8} (\theta_1^4 +\theta_3^4)=-\frac{1}{8}-3\sum\limits_{n=1}^{\infty}\sum\limits_{\substack{d|n \\ d~{\rm odd}}} d q^{n/2}=-\frac{1}{8}-3q^{1/2}-3q+\cdots, \\
\varepsilon=\varepsilon_{2}(\tau)&=\frac{1}{16} \theta_1^4 \theta_3^4=\sum\limits_{n=1}^{\infty}\sum\limits_{\substack{d|n \\ n/d~{\rm odd}}} d^3 q^{n/2}=q^{1/2}+8q+\cdots,
\end{split}
\end{equation}
the equation (\ref{odefeq}) has the solution
\[
f_2(z, \tau)=2\pi\sqrt{-1}\frac{\theta(z, \tau)}{\theta^\prime (0,\tau)} \frac{\theta_2(0,\tau)}{\theta_2(z, \tau)}.
\]

Let $M$ be a $4k$-dimensional closed smooth oriented manifold. Let $\{\pm 2\pi \sqrt{-1} x_i, \ 1\leq i\leq 2k\}$ be the formal Chern roots of the complexification $T_\mathbb{C}M= TM\otimes \mathbb{C}$. Consider the two characteristic numbers
\begin{equation}\label{ell12Meq}
\begin{split}
Ell_1(M ,\tau)&=2^{2k}\Big\langle \prod_{i=1}^{2k}\frac{2\pi \sqrt{-1}x_i}{f_1(x_i, \tau)}, [M]\Big\rangle \in \mathbb{Q}[[q]],\\
Ell_2(M ,\tau)&=\Big\langle \prod_{i=1}^{2k}\frac{2\pi \sqrt{-1}x_i}{f_2(x_i, \tau)}, [M]\Big\rangle \in \mathbb{Q}[[q^{1/2}]].
\end{split}
\end{equation}

$Ell_1(M ,\tau), Ell_2(M ,\tau)$ can be written as signature and $\widehat{A}$-genus twisted by the Witten bundles. More precisely,  
let 
\be\widehat{A}(M)=\prod_{i=1}^{2k}\frac{\pi \sqrt{-1}x_i}{\sinh \pi\sqrt{-1}x_i} \ee
be the $\widehat{A}$-class and 
\be\widehat{L}(M)=\prod_{i=1}^{2k}\frac{2\pi \sqrt{-1}x_i}{\tanh \pi\sqrt{-1}x_i} \ee
the $\widehat{L}$-class. Let $E$ be a complex vector bundle on $M$. $\Big\langle\widehat{L}(M)\mathrm{ch}E, [M]\Big\rangle$ is equal to the index of the twisted signature operator $\mathrm{ind}(d_s\otimes E)={\rm Sig}(M, E).$ When $M$ is spin, $\Big\langle\widehat{A}(M)\mathrm{ch}E, [M]\Big\rangle$ is equal to the index of the twisted Atiyah-Singer Dirac operator $\mathrm{ind}(D\otimes E)$.
When twisted by bundles naturally constructed from the tangent bundle $TM$ of $M$, denote
\[
\begin{split}
\widehat{A}(M, T^{i}\otimes \Lambda^{j}\otimes S^{k}) &:=\widehat{A}(M, \otimes^{i}T_{\mathbb{C}}M\otimes \Lambda^j(T_{\mathbb{C}}M)\otimes S^k(T_{\mathbb{C}}M)),\\
{\rm Sig}(M, T^{i}\otimes \Lambda^{j}\otimes S^{k}) &:={\rm Sig}(M, \otimes^{i}T_{\mathbb{C}}M\otimes \Lambda^j(T_{\mathbb{C}}M)\otimes S^k(T_{\mathbb{C}}M)),
\end{split}
\]
where $\Lambda^j(T_{\mathbb{C}}M)$ and $S^k(T_{\mathbb{C}}M)$ are the $j$-th exterior and $k$-th symmetric powers of $T_\mathbb{C}M$ respectively.

For any complex variable $t$, let
$$\Lambda_t(E)=\mathbb{C}+tE+t^2\Lambda^2(E)+\cdots ,
\ \  S_t(E)=\mathbb{C}+tE+t^2S^2(E)+\cdots$$  denote respectively
the total exterior and symmetric powers  of $E$, which live in
$K(M)[[t]]$. Denote by
\be \label{W1}\Theta_1(T_\CC M)=\bigotimes_{n=1}^\infty S_{q^n}(T_\CC M-\CC^{4k})
\otimes \bigotimes_{m=1}^\infty \Lambda_{q^m}(T_\CC M-\CC^{4k}), \ee
\be\label{W2}\Theta_2(T_\CC M)=\bigotimes_{n=1}^\infty
S_{q^n}(T_\CC M-\CC^{4k}) \otimes \bigotimes_{m=1}^\infty
\Lambda_{-q^{m-{1\over 2}}}(T_\CC M-\CC^{4k})\ee the Witten bundles, which are
elements in $K(M)[[q^{1\over2}]]$. Then one has
\be \label{expression}
\begin{split}
Ell_1(M ,\tau)&=\Big\langle\widehat{L}(M)\mathrm{ch}(\Theta_1(T_\CC M)), [M]\Big\rangle,\\
Ell_2(M ,\tau)&=\Big\langle \widehat{A}(M)\mathrm{ch}(\Theta_2(T_\CC M)), [M]\Big\rangle.
\end{split}
\ee

%----------------------------------------------------------------------------------------------------------------------------------------------------------------------------------------------------------%
\section{Proof of Theorem \ref{ellipticinjthm}}
 \label{sec: pfthm}

Let $M$ be a $4k$-dimensional closed smooth oriented manifold. By (\ref{oriented}), 
\be \label{a03defeq}
\phi(M)=a_0(M)\delta^6+a_1(M)\delta^4\varepsilon+a_2(M)\delta^2\varepsilon^2+a_3(M)\varepsilon^3,\ee
where $a_i(M)\in \mathbb{Z}\big[\frac{1}{2}\big] $, $0\leq i \leq 3$. First we show that one can express the $4$ Pontryagin numbers $a_i(M)$ ($0\leq i \leq 3$) in terms of $\widehat{A}$-genus, signature and their twists by the tangent bundle. 
\begin{proposition}\label{soai=ahatlemma} Let $M$ be a $4k$-dimensional closed smooth oriented manifold. One has
\[
\begin{split}
&a_0(M)=2^{18}\widehat{A}(M), \\ 
&a_1(M)=-2^{15}\cdot 3\cdot 5\widehat{A}(M)-2^{12}\widehat{A}(M, T), \\ 
&a_2(M)=2^{16}\cdot 3\widehat{A}(M)+2^{13}\widehat{A}(M, T)+\frac{1}{2^5}{\rm Sig}(M, T), \\
&a_3(M)=2^{15}\widehat{A}(M)-2^{12}\widehat{A}(M, T)-\frac{1}{2^5}{\rm Sig}(M, T)+ {\rm Sig}(M).
\end{split}
\]
\end{proposition}
\begin{proof} From the preliminary in Section \ref{sec: ell}, we see that
\begin{equation}\label{ell24termeq}
Ell_1(M)=2^{12}\big(a_0(M)\delta_1^6+a_1(M)\delta_1^4 \varepsilon_1+ a_2(M)\delta_1^2\varepsilon_1^2+ a_3(M)\varepsilon_1^3\big).
\end{equation}
and 
\begin{equation}\label{ell14termeq}
Ell_2(M)=a_0(M)\delta_2^6+a_1(M)\delta_2^4 \varepsilon_2+ a_2(M)\delta_2^2\varepsilon_2^2+ a_3(M)\varepsilon_2^3.
\end{equation}
On the other hand, by (\ref{W1}), (\ref{W2}) and (\ref{expression}), it is not hard to compute that
\begin{equation}\label{ellexpeq}
\begin{split}
Ell_1(M,\tau)&={\rm Sig}(M)+ \big(2{\rm Sig}(M, T)-48{\rm Sig}(M)\big)q+\cdots,\\
Ell_2(M,\tau)&=\widehat{A}(M)- \big(\widehat{A}(M, T)-24\widehat{A}(M)\big)q^{1/2}+\cdots .
\end{split}
\end{equation}
With the help of (\ref{d1e1eq}) and (\ref{d2e2eq}), we can compare (\ref{ellexpeq}) with (\ref{ell14termeq}) and (\ref{ell24termeq}). For instance, by modulo higher terms $(q^{\frac{1}{2}})^{i}$ with $i\geq 2$,
\[
\begin{split}
Ell_2(M)
&\equiv \frac{a_0(M)}{8^6}(-1-24q^{1/2})^6+\frac{a_1(M)}{8^4}(-1-24q^{1/2})^4 (q^{1/2})\\
&\equiv \frac{a_0(M)}{2^{18}}+( \frac{3^2 a_0(M)}{2^{14}}+\frac{a_1(M)}{2^{12}}) q^{1/2}.
\end{split}
\]
Combining the above formula with (\ref{ellexpeq}), we have
\begin{equation}\label{4t=noeq2}
\frac{a_0(M)}{2^{18}}=\widehat{A}(M), \ \ \frac{3^2 a_0(M)}{2^{14}}+\frac{a_1(M)}{2^{12}}=-\widehat{A}(M, T)+24\widehat{A}(M).
\end{equation}
Similarly, by modulo higher terms $q^{i}$ with $i\geq 2$,
\[
\begin{split}
Ell_1(M)
&\equiv 2^{12}\big( \frac{a_0(M)}{8^6}(2+48q)^6+\frac{a_1(M)}{8^4}(2+48q)^4 (\frac{1}{16}-q) \\
&\ \ \ +\frac{a_2(M)}{8^2}(2+48q)^2 (\frac{1}{16}-q)^2 +a_3(M)(\frac{1}{16}-q)^3\big)\\
&\equiv (a_0(M)+ a_1(M)+a_2(M)+a_3(M))\\
&\ \ \ +(144 a_0(M)+80a_1(M)+16a_2(M)-48a_3(M)) q.
\end{split}
\]
Combining the above formula with (\ref{ellexpeq}), we have
\begin{equation}\label{4t=noeq1}
\begin{split}
&a_0(M)+ a_1(M)+a_2(M)+a_3(M)={\rm Sig}(M), \\ 
&144 a_0(M)+80a_1(M)+16a_2(M)-48a_3(M)=2{\rm Sig}(M, T)-48{\rm Sig}(M).
\end{split}
\end{equation}
The equalities in (\ref{4t=noeq2}) and (\ref{4t=noeq1}) can be organized to result in a matrix equation
\begin{equation}\label{mxzi=no1eq}
\renewcommand\arraystretch{1.3}
\begin{pmatrix}
\frac{1}{2^{18}} & 0 & 0 &0 \\
\frac{3^2}{2^{14}} &\frac{1}{2^{12}}& 0 &0 \\
1& 1& 1 & 1\\
144 &80& 16 &-48\\
\end{pmatrix}
\cdot 
\begin{pmatrix}
a_0(M)\\   a_1(M) \\  a_2(M) \\  a_3(M) 
\end{pmatrix}
=
\begin{pmatrix}
\widehat{A}(M) \\  -\widehat{A}(M, T)+24\widehat{A}(M) \\ {\rm Sig}(M) \\  2{\rm Sig}(M, T)-48{\rm Sig}(M)
\end{pmatrix}.
\end{equation}
We can solve $a_i(M)$ from (\ref{mxzi=no1eq}), and then the proposition is proved.
\end{proof}

Now suppose $M$ is further a string manifold. With the string condition, we can rewrite the equalities of $a_i(M)$ in Proposition \ref{soai=ahatlemma} in terms of a new family of (twisted) genera, which is helpful for proving Theorem \ref{ellipticinjthm}.
\begin{proposition}\label{stringai=ahatprop} Let $M$ be a $24$-dimensional closed smooth string manifold. Then 
\begin{equation}\label{a=mxnoeq}
 \renewcommand\arraystretch{1.3}
\begin{pmatrix}
 a_0(M) \\   a_1(M) \\  a_2(M) \\  a_3(M) 
\end{pmatrix}
=
\begin{pmatrix}
2^{18} & 0 & 0 & 0\\
-2^{15}\cdot 3\cdot 5& -2^{15}\cdot 3 &0 & 0\\
2^{8}\cdot 3\cdot 331 & 2^{9}\cdot 3^5 & 2^{6} & 0 \\
-2^{8} \cdot 97 & -2^{9}\cdot 3\cdot 17 & -2^{6} &2^{3}
\end{pmatrix}
\cdot 
\begin{pmatrix}
\widehat{A}(M)\\ \frac{1}{24}\widehat{A}(M, T) \\ \widehat{A}(M,\Lambda^2) \\ \frac{1}{8}{\rm Sig}(M)
\end{pmatrix}.
\end{equation} 
\end{proposition}
\begin{proof}
Under the string condition, the twisted and untwisted genera in Proposition \ref{soai=ahatlemma} possess intrinsic relations. Indeed, by combining modularity of the Witten genus and a modular form constructed by Liu and Wang in \cite{LW13}, Chen and Han \cite{CH15} showed that, when $M$ is a $24$-dimensional closed smooth string manifold, one has
\begin{equation}\label{ch15eq} 
{\rm Sig}(M, T)= 2^{11}\big(\widehat{A}(M, \Lambda^2)-47 \widehat{A}(M, T)+900 \widehat{A}(M)\big).
\end{equation}
With (\ref{ch15eq}) we can rewrite the equalities of $a_i(M)$ in Proposition \ref{soai=ahatlemma} as displayed in this proposition.
\end{proof}

In \cite{HH21} we have determined an integral basis of $\Omega_{24}^{String}$, which is crucial for the proof of Theorem \ref{ellipticinjthm}.
 \begin{theorem}[Theorem 1 and Corollary 3 in \cite{HH21}]\label{basisthm}
The correspondence 
$\kappa :\Omega
_{24}^{String}\rightarrow $ $\mathbb{Z}\oplus \mathbb{Z}\oplus \mathbb{Z}
\oplus \mathbb{Z}$ defined by
\[
\kappa (M)=(\widehat{A}(M),\frac{1}{24}\widehat{A}(M,T),
\widehat{A}(M,\Lambda ^{2}),\frac{1}{8}{\rm Sig}(M))
\]
is an isomorphism of abelian groups. Moreover, there exists a basis $\{M_{i}\}_{1\leq i\leq 4}$ of $\Omega
_{24}^{String}$ such that
\[
\hspace{1.35cm}
 \renewcommand\arraystretch{1.3}
K:=
\begin{pmatrix}
\kappa (M_{1})\\
\kappa (M_{2})\\
\kappa (M_{3})\\
\kappa (M_{4})
\end{pmatrix}^\tau
=
\begin{pmatrix}
0 & 1  & 0 & 0\\
-1 & 0 &0 & 0\\
2^3\cdot 3^3\cdot 5  & 2^2\cdot 3\cdot 17 \cdot 1069 & -1 & 0\\
2^8\cdot 3\cdot 61   & 2^8\cdot 5\cdot 37 & 2^2\cdot 7 & 1
\end{pmatrix}.\hspace{1.35cm}\Box
\]
\end{theorem}

Now we are ready to prove Theorem \ref{ellipticinjthm}.
\begin{proof}[Proof of Theorem \ref{ellipticinjthm}]
Suppose $M$ satisfies that $\phi(M)=0$. Then by (\ref{a03defeq}) $a_i(M)=0$ for $0\leq i\leq 3$. Notice that in (\ref{a=mxnoeq}) the coefficient matrix is invertible.
Then by Proposition \ref{stringai=ahatprop}, the $4$ index numbers $\widehat{A}(M)$, $\frac{1}{24}\widehat{A}(M, T)$, $\widehat{A}(M,\Lambda^2)$ and $\frac{1}{8}{\rm Sig}(M)$ vanish.
This means that $\kappa(M)=(0,0,0,0)$. Since by Theorem \ref{basisthm} $\kappa$ is an isomorphism, $[M]=0\in \Omega_{24}^{String}$. Hence $\phi$ is injective.

To compute the image of the elliptic genus, we need to compute the elliptic genus of the generators $M_i$ ($1\leq i\leq 4$) in Theorem \ref{basisthm}. By Proposition \ref{stringai=ahatprop} and Theorem \ref{basisthm}, they can be computed by matrix multiplication
\[
\begin{split}
\renewcommand\arraystretch{1.3}
& \begin{pmatrix}
2^{18} & 0 & 0 & 0\\
-2^{15}\cdot 3\cdot 5& -2^{15}\cdot 3 &0 & 0\\
2^{8}\cdot 3\cdot 331 & 2^{9}\cdot 3^5 & 2^{6} & 0 \\
-2^{8} \cdot 97 & -2^{9}\cdot 3\cdot 17 & -2^{6} &2^{3}
\end{pmatrix}
\cdot 
\begin{pmatrix}
0 & 1  & 0 & 0\\
-1 & 0 &0 & 0\\
2^3\cdot 3^3\cdot 5  & 2^2\cdot 3\cdot 17 \cdot 1069 & -1 & 0\\
2^8\cdot 3\cdot 61   & 2^8\cdot 5\cdot 37 & 2^2\cdot 7 & 1
\end{pmatrix}
\\
& =
\begin{pmatrix}
0 & 2^{18}  & 0 & 0\\
2^{15}\cdot 3 & -2^{15}\cdot 3\cdot 5 &0 & 0\\
-2^{11}\cdot 3^3 & 2^{11}\cdot 3^3\cdot 257& -2^6 & 0\\
2^{12}\cdot 3^4   & -2^{12}\cdot 3^4\cdot 41 & 2^5\cdot 3^2 & 2^3
\end{pmatrix}
= (a_j(M_i))_{j\times i}.
\end{split}
\]
Combining (\ref{a03defeq}), the above matrix gives the $4$ generators of the image $\phi(\Omega_\ast^{String})$ as 
\[
\begin{split}
 2^3 \cdot \varepsilon^3=\phi(M_4),& \\
- (8\delta)^2\varepsilon^2+ 2^5\cdot 3^2 \cdot \varepsilon^3=\phi(M_3), &\\
2^{3}\cdot 3 \cdot (8\delta)^4 \varepsilon -2^{5}\cdot 3^3 \cdot (8\delta)^2\varepsilon^2 +2^{12}\cdot 3^4  \cdot  \varepsilon^3=\phi(M_1),&\\
 (8\delta)^6-2^{3}\cdot 3\cdot 5 \cdot (8\delta)^4 \varepsilon +2^{5}\cdot 3^3\cdot 257 \cdot (8\delta)^2\varepsilon^2-2^{12}\cdot 3^4\cdot 41\cdot  \varepsilon^3=\phi(M_2).&
\end{split}
\]
It follows that $\phi(\Omega_\ast^{String})$ is generated by $2^3 \cdot\varepsilon^3$, $(8\delta)^2\varepsilon^2$,
$2^{3}\cdot 3 \cdot(8\delta)^4 \varepsilon$ and $(8\delta)^6$.
This completes the proof of the theorem.
\end{proof}
 
%----------------------------------------------------------------------------------------------------------------------------------------------------------------------------------------------------------%

\end{document}